\documentclass[12pt,leqno,twoside,sort]{amsart}

\usepackage{amssymb,amsmath,amsthm,soul,color}
\usepackage[normalem]{ulem}
\usepackage{amssymb,amsmath,amsthm}
\usepackage{mathrsfs}
\usepackage[utf8]{inputenc}
\usepackage{todonotes}
\usepackage[left=2cm, right=2cm, top=2cm, bottom=2cm]{geometry}
\usepackage{url}
\usepackage[numbers,sort&compress]{natbib}
\usepackage{fixmath}

\usepackage[T1]{fontenc}
\usepackage{lmodern}
\usepackage{microtype}

\usepackage[colorlinks=true,urlcolor=blue,
citecolor=red,linkcolor=blue,linktocpage,pdfpagelabels,
bookmarksnumbered,bookmarksopen]{hyperref}

\usepackage{mathtools}
\usepackage{xcolor,cancel}

\newtheorem{Th}{Theorem}[section]

\newtheorem{Lem}[Th]{Lemma}
\newtheorem{Cor}[Th]{Corollary}

\newtheorem{Rem}[Th]{Remark}

\newcommand{\vp}{\varphi}

\newcommand{\eps}{\varepsilon}

\newcommand{\R}{\mathbb{R}}

\newcommand{\cB}{{\mathcal B}}
\newcommand{\cC}{{\mathcal C}}
\newcommand{\cD}{{\mathcal D}}
\newcommand{\cE}{{\mathcal E}}

\newcommand{\cG}{{\mathcal G}}
\newcommand{\cH}{{\mathcal H}}
\newcommand{\cI}{{\mathcal I}}
\newcommand{\cJ}{{\mathcal J}}

\newcommand{\cM}{{\mathcal M}}
\newcommand{\cN}{{\mathcal N}}
\newcommand{\cO}{{\mathcal O}}
\newcommand{\cP}{{\mathcal P}}

\newcommand{\cS}{{\mathcal S}}

\newcommand{\U}{{\mathbf U}}

\newcommand{\weakto}{\rightharpoonup}

\newcommand{\curl}{\nabla \times}
\renewcommand{\div}{\mathrm{div}\,}

\numberwithin{equation}{section}

\newcommand{\supp}{\mathrm{supp}\,}

\begin{document}

\title{Semiclassical states for the curl-curl problem}

\author[B. Bieganowski]{Bartosz Bieganowski}
\address[B. Bieganowski]{\newline\indent
			Faculty of Mathematics, Informatics and Mechanics, \newline\indent
			University of Warsaw, \newline\indent
			ul. Banacha 2, 02-097 Warsaw, Poland}	
			\email{\href{mailto:bartoszb@mimuw.edu.pl}{bartoszb@mimuw.edu.pl}}	

\author[A. Konysz]{Adam Konysz}
\address[A. Konysz]{\newline\indent  	
			Faculty of Mathematics and Computer Science,		\newline\indent 
			Nicolaus Copernicus University, \newline\indent 
			ul. Chopina 12/18, 87-100 Toru\'n, Poland}
			\email{\href{mailto:adamkon@mat.umk.pl}{adamkon@mat.umk.pl}}
			
\author[J. Mederski]{Jarosław Mederski}
\address[J. Mederski]{\newline\indent  	
			Institute of Mathematics,		\newline\indent 
			Polish Academy of Sciences, \newline\indent 
			ul. \'Sniadeckich 8, 00-656 Warsaw, Poland}
			\email{\href{mailto:jmederski@impan.pl}{jmederski@impan.pl}}

\date{}	\date{\today}

\begin{abstract} 
We show the existence of the so-called semiclassical states $\U:\R^3\to\R^3$ to the following curl-curl problem
$$
\varepsilon^2\; \curl (\curl \U) + V(x) \U = g(\U),
$$
for sufficiently small $\varepsilon > 0$. We study the asymptotic behaviour of solutions as $\eps\to 0^+$ and we investigate also a related nonlinear Schr\"odinger equation involving a singular potential. The problem models large permeability nonlinear materials satisfying the system of Maxwell equations.

\medskip

   
\end{abstract}

\maketitle

\pagestyle{myheadings} \markboth{\underline{B. Bieganowski, A. Konysz, J. Mederski}}{
		\underline{Semiclassical states for the curl-curl problem}}

\section{Introduction}

We look for time-harmonic wave field solving the system of {\em Maxwell equations} of the form
$$
\left\{ \begin{array}{l}
\curl \cH = \partial_t \cD,\\
\div (\cD) = 0,\\
\partial_t\cB+\curl \cE = 0,\\
\div (\cB) = 0,
\end{array} \right.
$$
where $\cE$ is the electric field, $\cB$ is the magnetic field, $\cD$ is the electric displacement field and $\cH$ denotes the magnetic induction.  In the absence of charges, currents and magnetization, we consider also the following constitutive relations ({\em material laws})
$$
\begin{cases}
\cD=\epsilon(x) \cE + \cP_{NL}, \\
\cH=\mu^{-1} \cB,
\end{cases}
$$
where $\cP_{NL}=\chi(\langle |\cE|^2\rangle)\cE$ is the nonlinear polarization, $\langle |\cE(x)|^2\rangle = \frac{1}{T}\int_0^T |\cE(x)|^2\,dt$ is the average intensity of a time-harmonic electric field over one period $T=2\pi/\omega$, $\epsilon(x)\in\R$ is the permittivity of the medium, $\mu>0$ is the constant magnetic permeability, and $\chi$ is the scalar nonlinear susceptibility which depends on the time averaged intensity of $\cE$ only. For instance, the probably most common type of nonlinearity in the physics and engineering literature, is the {\em Kerr nonlinearity} of the form
$\chi(\langle |\cE|^2 \rangle)  = \chi^{(3)}\langle |\cE|^2\rangle$, but we will able to treat a more general class of nonlinear phenomena.

Such situations were widely studied from the physical and mathematical point of view \cite{Stuart91,Stuart:1993,FundPhotonics} and recall that
taking the curl of Faraday's law, i.e. the third equation in the Maxwell system, and inserting the material laws together with Amp\'{e}re's law we find that $\cE$ has to satisfy the {\em nonlinear electromagnetic wave equation}
\begin{equation*} 
\nabla\times\big(\mu^{-1}\nabla \times \cE\big) + \partial_{tt}\left(\epsilon(x) \cE + \chi(\langle |\cE|^2\rangle)\cE\right)=0\quad\hbox{for }(x,t)\in\R^3\times\R.
\end{equation*}
Looking for time-harmonic fields of the form $\cE(x,t)= \mathbf{U}(x) \cos(\omega t)$, $\U:\R^3\to\R^3$, the above equation leads to the {\em curl-curl problem}
\begin{equation}\label{eq:maxwell}
\mu^{-1} \curl (\curl \mathbf{U}) + V(x)\mathbf{U} = g(\mathbf{U}), \quad x \in \R^3
\end{equation}
with $V(x):= - \omega^2 \epsilon(x)$ and $g(\mathbf{U}):=\omega^2\chi\big(\frac12|\mathbf{U}|^2\big)$. Note that having solved \eqref{eq:maxwell}, hence also the nonlinear electromagnetic wave equation, one obtains the electric displacement field $\cD$ directly from the constitutive relations and the magnetic induction
$\cB$ may be obtained by time integrating Faraday's law with divergence free initial condition. Moreover, we also get the magnetic field 
$\cH=\mu^{-1} \cB$. Altogether, we find {\em exact propagation} of the electromagnetic field in the nonlinear medium according to the Maxwell equations  with the time-averaged material law,  see also \cite{Stuart91,Stuart:1993,FundPhotonics,Med,BM1}. It is worth mentioning that the exact propagation in nonlinear optics plays a crucial role and, e.g. cannot by studied by approximated models, see \cite{Akhmediev-etal, Ciattoni-etal:2005} and references therein. Therefore,  in this paper, we are interested in exact time-harmonic solutions of the Maxwell equations.

The nonlinear curl-curl problem \eqref{eq:maxwell} has been recently studied e.g. in \cite{BM1,BM2} on a bounded domain and in \cite{BDPR,Med,MSSz} on $\R^3$, see also the survey \cite{MedSch} and references therein. 
In all these works the asymptotic role of the magnetic permeability was irrelevant from the mathematical point of view and therefore it was assumed that $\mu=1$, or on a bounded domain $\mu$ was a bounded $3\times 3$-tensor \cite{BM2,BM3}. In the present paper we study the asymptotic behaviour of the problem with permeability $\mu\to\infty$, and simultaneously we admit a wide range of permittivity expressed in terms of $V\in \cC(\R^3)$ as follows:
\begin{enumerate}
	\item [(V1)] $0<V_0 := \inf V\leq V(0)<V_\infty\leq \liminf_{|x|\to+\infty} V(x)$
\end{enumerate}
for some $V_\infty\in\R$ and the last limit may be infinite.
In the physics literature, the positive extremely large permeability in magnetic
materials is usually due to the formation of magnetic domains \cite{Grimberg,Chew}, while (V1) models the so-called 
epsilon-negative materials \cite{Vanbesien,Chew}.

From the mathematical point of view, setting $\varepsilon^2 :=\mu^{-1}$ in \eqref{eq:maxwell}, since $\varepsilon^2 \curl ( \curl \U)=\curl ( \curl \mathbf{U}(\varepsilon \cdot))$ and replacing $\mathbf{U}(\varepsilon \cdot)$ by $\mathbf{U}$ we end up with the following problem
\begin{equation}\label{eq:curlcurl}
\curl (\curl \U) + V_\varepsilon(x) \U = g(\U),
\end{equation}
where $V_\varepsilon (x) := V(\varepsilon x)$, and $G:\R^3\to\R$ is responsible for the nonlinear effect and $g:=\nabla G$. From now on we do not use the notation of the permittivity $\epsilon(x)$. Our aim is to investigate \eqref{eq:curlcurl} in the limit $\eps\to 0^+$.

Due to the strongly indefinite nature of \eqref{eq:curlcurl}, e.g. the curl-curl operator $\curl (\curl \cdot)$ contains an infinite dimensional kernel, we introduce the cylindrical symmetry and, as in \cite{GMS}
we look for solutions of the form 
\begin{equation}\label{eq:Uu}
\U(x) = \frac{u(r, x_3)}{r} \left( \begin{array}{c} -x_2 \\ x_1 \\ 0 \end{array} \right), \quad r = \sqrt{x_1^2+x_2^2},\; x=(x_1,x_2,x_3),
\end{equation}
which leads to the following Schr\"odinger equation
\begin{equation}\label{eq:singular}
-\Delta u +\frac{u}{|y|^2}+V_\varepsilon (x)u=f(u)\quad
\hbox{ for } x = (y,z) \in \R^N = \R^{K} \times \R^{N-K}
\end{equation}
with $N=3$, $K=2$ and $g(\alpha w)=f(\alpha)w$ for $\alpha\in\R$, $w\in \R^3$ such that $|w|=1$.

In what follows, $\lesssim$ denotes the inequality up to a multiplicative constant.

In general, let $N\geq 3$, $2^*=\frac{2N}{N-2}$, and we consider the following assumptions on $f$.
\begin{enumerate}
\item[(F1)] $f : \R \rightarrow \R$ is continuous and there is $p \in (2, 2^*)$ such that
$$
|f(u)| \lesssim 1 + |u|^{p-1}.
$$
\item[(F2)] $f(u)=o(u)$ as $u\to 0$.
\item[(F3)]$\frac{F(u)}{u^2}\to+\infty$ as $|u|\to\infty$, where $F(u) := \int_0^u f(s) \, ds$.
\item[(F4)] $\frac{f(u)}{|u|}$ is increasing on $(-\infty,0)$ and on $(0,\infty).$
\end{enumerate}
In a similar way as in \cite[Theorem 2.1]{GMS} (cf. \cite{BB,BMS}) weak solutions  to \eqref{eq:singular} correspond to weak solutions of the form \eqref{eq:Uu} to \eqref{eq:curlcurl}. Clearly, concerning the Kerr nonlinearity one has $f(u)=\frac12\chi^{(3)}|u|^2u$, $\chi^{(3)}>0$, $N=3$, and the above assumptions are satisfied.

Let $\cO(K)$ denote the orthogonal group acting on $\R^K$, $K\geq 2$, and let  $\cG(K) := \cO(K) \times I_{N-K}\subset \cO(N)$ for $N>K\geq 2$.  Let $V\in \cC^{\cG(K)}(\R^N)$ be a continuous potential invariant with respect to $\cG(K)$.
The first main result reads as follows.

\begin{Th}\label{thm:main1}
Suppose that $V\in\cC^{\cG(K)}(\R^N)$, $N>K\geq 2$, and (V1), (F1)--(F4) hold. Then there exists $\varepsilon_0>0$ such that for any $\varepsilon\in (0,\varepsilon_0)$, \eqref{eq:singular} has a nontrivial weak solution $u_\varepsilon$, which is invariant with respect to $\cG(K)$. Moreover, if $f$ is odd, then $u_\eps\in L^{\infty}(\R^N)$ is nonnegative and
$$
\limsup_{|x|\to\infty} |x|^\nu u_\eps (x) = 0
$$
for any $\nu < \frac{N-2+\sqrt{(N-2)^2+4}}{2}$.
\end{Th}

A weak solution to \eqref{eq:singular} is a critical point of the energy functional
$\cJ_\varepsilon \colon X_\eps \rightarrow \R$:
\begin{equation}\label{eq:J-eps}
\cJ_\varepsilon (u):=\frac12\int_{\R^N}|\nabla u|^2+\frac{u^2}{|y|^2}+V_\varepsilon (x)u^2 \,dx-\int_{\R^N}F(u) \, dx
\end{equation}
defined on 
$$X_\eps:=\Big\{u\in H^1(\R^N)\ : \ \int_{\R^N}\frac{u^2}{|y|^2}+V_\eps(x)u^2\,dx<\infty\Big\}.$$ 
Recall that solutions to \eqref{eq:singular} with $\eps\to 0^+$ are the so-called  \textit{semiclassical states}. Recently many papers have been devoted to study semiclassical states for the Schr\"odinger equation, see eg. \cite{delPinoCV, ByeonJeanjean,ByeonTanaka,delPinoJFA,dAvenia,Rabinowitz,Wang}  and references therein, however the usual techniques are difficult to apply to the Schr\"odinger operator involving the singular potential, since we are not able to apply the regularity results or $L^\infty$-elliptic estimates. As we shall see, we demonstrate an extension of the classical approach due to Rabinowitz \cite{Rabinowitz} to prove Theorem \ref{thm:main1}. Finally we recall that solutions to \eqref{eq:singular} with $V_\eps \equiv 0$ have been recently obtained by Badiale et. al. \cite{BadialeBenciRolando} with a different set of growth assumptions imposed on $f$, e.g. supercritical growth at $0$, excluding the Kerr nonlinearity, cf. \cite{GMS}.

In order to study the asymptotic behaviour of $u_\varepsilon$ we introduce the following assumptions.
\begin{enumerate}
\item[(V2)] $\lim_{|x| \to \infty} V(x) =V_\infty<\infty$.
\item[(V3)] $V$ is H\"older continuous at $0$ with some exponent $\alpha>0$. 
\end{enumerate}
Observe that the continuity of $V$ and (V2) imply  that $V \in L^\infty (\R^N)$ and $X_\eps$ does not depend on $\eps$.

\begin{Th}\label{thm:main2}
Suppose that $V\in \cC^{\cG(K)}(\R^N)$, (V1)--(V3), (F1)--(F4) hold and $f$ is odd. Then, there is a sequence $\varepsilon_n \to 0$ such that one of the following holds. Either
\begin{itemize}
\item[(a)] there is a nontrivial weak solution $U$ to \eqref{eq:limiting} with $k =V_\infty$ (i.e. \eqref{eq:singular} with $V_\eps\equiv V_\infty$) that
$$
u_{\varepsilon_n} - U(\cdot - (0,z_n)) \to 0 \quad \hbox{ in } X_1 \hbox{ and in } L^p (\R^N)
$$
for some translations $(z_n) \subset \R^{N-K}$ satisfying $\varepsilon_n |z_n| \to \infty$; 
\end{itemize}
or
\begin{itemize}
\item[(b)] there is $\ell \geq 1$, such that for all $j\in\{1,..,\ell\}$ there exist $(z_n^j) \subset \R^{N-K}$ and nontrivial weak solutions $U_{j}$ to \eqref{eq:limiting} with $k=V(0, z^j) $ for some $z^j \in \R^{N-K}$, such that
$$
u_{\varepsilon_n} - \sum_{j=1}^\ell U_j ( \cdot - (0, z_n^j) ) \to 0 \quad \hbox{ in } X_1 \hbox{ and in } L^p (\R^N);
$$
moreover $z^j = \lim_{n\to\infty} \varepsilon_n z_n^j$ and $\ell \leq \frac{m_{V_\infty}}{m_{V_0}}$, where $m_{V_\infty}$, $m_{V_0}$ are defined in \eqref{eq:m}.
\end{itemize}
\end{Th}

Using the correspondence between weak solutions to \eqref{eq:curlcurl} and \eqref{eq:singular} (cf. \cite{BB, GMS}) we obtain the following result.

\begin{Th}\label{thm:main3}
Suppose that $N=3$, $K=2$, $V\in \cC^{\cG(2)}(\R^3)$, (V1)--(V3), (F1)--(F4) hold, $g(\alpha w)=f(\alpha)w$ for $\alpha\in\R$, $w\in \R^3$ such that $|w|=1$ (in particular, $f$ is odd). Then, for sufficiently small $\varepsilon$ there are weak solutions $\mathbf{U}_\eps$ to \eqref{eq:curlcurl} of the form \eqref{eq:Uu}; $\U_\eps \in L^\infty (\R^3;\R^3)$ and
$$
\limsup_{|x|\to\infty} |x|^\nu | \U_\eps (x) | = 0 \quad \hbox{for every } \nu <  \frac{N-2+\sqrt{(N-2)^2+4}}{2}.
$$
Moreover, there is a sequence $\varepsilon_n \to 0^+$ such that one of the following holds. Either
\begin{itemize}
\item[(a)] there is a nontrivial weak solution $\U$ to \eqref{eq:curlcurl} with $V_\eps \equiv V_\infty$ such that
$$
\U_{\varepsilon_n} - \U(\cdot - (0,z_n)) \to 0 \quad \hbox{ in } H^1 (\R^3; \R^3)
$$
for some translations $(z_n) \subset \R$ satisfying $\varepsilon_n |z_n| \to \infty$; 
\end{itemize}
or
\begin{itemize}
\item[(b)] there is $\ell \geq 1$, such that for all $j\in\{1,..,\ell\}$ there exist $(z_n^j) \subset \R$ and nontrivial weak solutions $\U_{j}$ to \eqref{eq:curlcurl} with $V_\varepsilon \equiv V(0, z^j) $ for some $z^j \in \R$, such that
$$
\U_{\varepsilon_n} - \sum_{j=1}^\ell \U_j ( \cdot - (0, z_n^j) ) \to 0 \quad \hbox{ in } H^1 (\R^3; \R^3);
$$
moreover $z^j = \lim_{n\to\infty} \varepsilon_n z_n^j$.
\end{itemize}
\end{Th}

\section{Functional setting}

We consider the group action of $\cG(K)$ on $H^1 (\R^N)$. Then, by $H^1_{\cG(K)} (\R^N)$ we denote a subspace of $\cG(K)$-invariant functions from $H^1(\R^N)$. In Sections 2--5 we always assume that $V\in\cC^{\cG(K)}(\R^N)$, $N>K\geq 2$.

Let 
$$
X_\varepsilon^{\cG(K)} := X_\varepsilon \cap H^1_{\cG(K)}(\R^N).
$$
The norm in $X_\eps$ and in  $X_\varepsilon^{\cG(K)}$ is given by
\[
\|u\|^2_\eps:=\int_{\R^N}|\nabla u|^2+\frac{u^2}{|y|^2}+V_\eps(x)u^2\,dx.
\]
Note that, under (V1),
$$
\| u\|^2_\eps \geq \int_{\R^N} |\nabla u|^2 + V_\eps(x) u^2 \, dx \geq \int_{\R^N} |\nabla u|^2 + V_0 u^2 \, dx
$$
and therefore embeddings
$$
X_\varepsilon^{\cG(K)} \subset H^1_{\cG(K)} (\R^N) \subset L^s (\R^N)
$$
are continuous, where $2 \leq s \leq 2^*$.

For every $\varepsilon > 0$, the functional $\cJ_\varepsilon : X_\eps \rightarrow \R$ associated with \eqref{eq:singular} is, under (F1) and (V1), of $\cC^1$-class and its critical points are weak solutions to \eqref{eq:singular}. Note that, thanks to the Palais' principle of symmetric criticality (see \cite{Palais}), every critical point of $\cJ_\eps$ restricted to $X_\varepsilon^{\cG(K)}$ is also a critical point of the free functional, and therefore, a weak solution to \eqref{eq:singular}. We will work on the following Nehari manifold
\[\cN_\varepsilon = \left\{ u\in X_\varepsilon^{\cG(K)} \setminus \{0\}: \int_{\R^N} |\nabla u|^2 + \frac{u^2}{|y|^2} + V_\varepsilon(x) u^2 \, dx =\int_{\R^N} f(u)u\,dx \right\},\]
and we define
$$
c_\varepsilon := \inf_{\cN_\varepsilon} \cJ_\varepsilon.
$$
Observe that, if $V \in L^\infty (\R^N)$, then $X_\varepsilon$ does not depend on $\varepsilon$ and $X_\varepsilon = Y$, where
$$
Y := \left\{ u \in H^1 (\R^N) \ : \ \int_{\R^N} \frac{u^2}{|y|^2} \, dx < \infty \right\}.
$$
We define $Y^{\cG(K)} := Y \cap H^1_{\cG(K)}(\R^N)$. In $Y$ we consider the norm
$$
\|u\|^2_Y := \int_{\R^N} |\nabla u|^2 + \frac{u^2}{|y|^2} + u^2 \, dx, \quad u \in Y.
$$
It is natural to consider the limiting problem of the form
\begin{equation}\label{eq:limiting}
-\Delta u +\frac{u}{|y|^2}+ k u=f(u)\quad
\hbox{for } x = (y,z) \in \R^N = \R^{K} \times \R^{N-K},
\end{equation}
where $k > 0$, and the corresponding energy functional $\Phi_k : Y \rightarrow \R$
\[
\Phi_k (u):=\frac12\int_{\R^N}|\nabla u|^2+\frac{u^2}{|y|^2}+k u^2 \,dx-\int_{\R^N}F(u)\, dx.
\]
Again, thanks to the Palais' principle of symmetric criticality, critical points of $\Phi_k$ restricted to $Y^{\cG(K)}$ are also critical points of the free functional. We set also
$$
\cM_k := \left\{ u \in Y^{\cG(K)} \setminus \{0\} \ : \ \int_{\R^N} |\nabla u|^2 + \frac{u^2}{|y|^2} + k u^2 \, dx = \int_{\R^N} f(u)u \, dx \right\}
$$
and
\begin{equation}\label{eq:m}
m_k := \inf_{\cM_k} \Phi_k.
\end{equation}

\section{Continuous dependence of Nehari manifold levels}

We start our analysis with the problem \eqref{eq:singular} with $\varepsilon = 1$. Hence, in this section, we will write for simplicity $X^{\cG(K)} := X_1^{\cG(K)}$, $\cJ := \cJ_1$, $\cN := \cN_{1}$, $c := c_1$. It is classical to check the following fact (cf. \cite{SzulkinWeth}).
\begin{Lem}
For every $u \in X^{\cG(K)} \setminus \{ 0 \}$ there exists unique $t_{V} (u)>0$ such that $t_V(u)u\in\cN$, 
\begin{equation}\label{maxonN}
    \cJ(t_V(u)u)=\max_{t \geq 0}\cJ(t u),
\end{equation}
$\cN$ is bounded away from zero, and $\widehat{m}_V : \cS \rightarrow \cN$ given by $\widehat{m}_V(u) := t_V(u) u$ is a homeomorphism, where $\cS$ is the unit sphere in $X^{\cG(K)}$.
\end{Lem}

\begin{Lem}\label{lem:c-tilde}
    Suppose that $V$, $\widetilde{V} \in L^\infty (\R^N)$ satisfy (V1). If $V\geq\widetilde{V}$ then $c\geq\widetilde{c}$, where $\widetilde{c} := \inf_{\widetilde{\cN}} \widetilde{\cJ}$, $\widetilde{\cJ}$ is the energy functional with $V$ replaced by $\widetilde{V}$ and $\widetilde{\cN}$ is the corresponding Nehari manifold in $Y^{\cG(K)}$.
\end{Lem}
\begin{proof}
    Note that for all $u\in \widetilde{\cN}$
    $$
    \widetilde{c} = \inf_{\widetilde{\cN}} \widetilde{\cJ} \leq \widetilde{\cJ} (u) \leq \cJ(u) \leq \cJ(t_V(u) u).
    $$
    Observe that $\widetilde{\cN} \ni u \mapsto \eta(u) := t_V(u)u \in \cN$ is a bijection, since $\eta(u) = \widehat{m}_V \circ \widehat{m}_{\widetilde{V}}^{-1}$. Hence
    $$
    \widetilde{c} \leq \cJ(v) \quad \mbox{ for any } v \in \cN.
    $$
    Thus $\widetilde{c} \leq c$ and the proof is completed.
\end{proof}

We will show the following continuous dependence of $c$ with respect to the potential $V$.

\begin{Th}\label{thm:cont_dep}
    Suppose that $V \in L^\infty (\R^N)$ and $(V_n) \subset L^\infty (\R^N)$ satisfy (V1). Then $c$ depends continuously on $V$ in $L^\infty$, i.e. if $V_n \to V$ in $L^\infty(\R^N)$ then $c(V_n)\to c(V)$, where $c(V)$ denotes the infimum on the corresponding Nehari manifold in $Y^{\cG(K)}$ of the energy functional with the potential $V$.
\end{Th}

\begin{proof}
Fix $\delta > 0$. Observe that for $n \gg 1$
\[
V+\delta \geq V+|V_n-V|\geq V_n\geq V-|V_n-V|\geq V-\delta,
\]
so having in mind Lemma \ref{lem:c-tilde}, it suffices to prove that
\[
c(V+h)\to c(V), \quad h \in \R, \ h \to 0.
\]
We will verify it first for $h < 0$ and $h \to 0^-$.
From Lemma \ref{lem:c-tilde} 
\[
\lim_{h\to0^-} c(V+h) = \underline{c} \leq  c(V).
\]
Suppose that
\begin{equation}\label{nierc}
\underline{c}<c(V).
\end{equation}
Define 
\[
\cI_h(u):=\frac12\int_{\R^N}|\nabla u|^2+\frac{u^2}{|y|^2}+\left(V(x)+h\right)u^2 \,dx-\int_{\R^N}F(u)\, dx = \cJ(u) + \frac12 h |u|_2^2, \quad u \in Y^{\cG(K)}.
\]
Here and in what follows $|\cdot|_q$ stands for the Lebesgue $L^q$-norm for $q\geq 1$.

From \cite[Theorem 2.1]{BBJM}, there is a bounded sequence $(u_n)\subset\cN_h$ such that  $\cI_h(u_n)\to c(V+h)$,
where $\cN_{h}$ is the Nehari manifold in $Y^{\cG(K)}$ corresponding to $\cI_h$.
Then 
\begin{align*}
c(V)&\leq \cJ(t_V(u_n)u_n)=\cI_h(t_V(u_n)u_n)-\frac12h t_V(u_n)^2 |u_n|_2^2\leq \cI_h(u_n)-\frac12h t_V(u_n)^2 |u_n|_2^2.
\end{align*}
Since $(u_n)$ is bounded in $Y^{\cG(K)}$, $|u_n|_2 \lesssim 1$. We will show that $t_n := t_V(u_n)$ is bounded. Suppose by contradiction that $t_n\to\infty$. Since $(u_n) \subset \cN_{h}$ we have $\liminf_{n \to \infty} |u_n|_p > 0$. Hence \cite[Corollary 3.32]{JM} implies that there is a sequence $(z_n)\subset \R^{N-K}$, $\beta >0$ and $R>0$ such that
\begin{equation*}
\liminf_{n\to\infty}\int_{B((0,z_n), R)}u_n^2\,dx>\beta,
\end{equation*}
and $u_n(\cdot - (0,z_n)) \weakto u \neq 0$. Observe that, thanks to (F3) and (F4), $t_n$ satisfies
\begin{align*}
1 \gtrsim \int_{\R^N} |\nabla u_n|^2 + \frac{u_n^2}{|y|^2} + V(x) u_n^2 \, dx &= \int_{\R^N} \frac{f(t_n u_n) u_n}{t_n} \, dx \geq 2 \int_{\R^N} \frac{F(t_n u_n)}{t_n^2} \, dx \\
&= 2 \int_{\R^N} \frac{F(t_n u_n(\cdot - (0,z_n)))}{t_n^2 |u_n(\cdot - (0,z_n))|^2} |u_n(\cdot - (0,z_n))|^2 \, dx \to \infty,
\end{align*}
which is a contradiction. Hence we can choose $h$ small enough to get contradiction with \eqref{nierc}. The reasoning for $h > 0$ is similar. Therefore $\lim_{h\to 0} c(V+h) = c(V)$ and the proof is completed.
\end{proof}

\section{The limiting problem}

In this section we are interested in the limiting problem \eqref{eq:limiting} and its connection to the problem with an external potential $V$. In what follows, $c := c_1$, $\cJ := \cJ_1$ and $\cN := \cN_1$.

We start with noting the following existence result, which can be obtain using standard techniques; namely using the Nehari manifold method connected with the concentration-compactness argument in the spirit of \cite[Corollary 3.2, Remark 3.2]{JM}, cf. \cite[Corollary 7.1]{BB}.

\begin{Th}\label{thm:limit}
Let $k>0$ and (F1)--(F4) hold. Then $m_k$ is a critical value of $\Phi_k$ with a corresponding weak solution $u_k$ of the problem \eqref{eq:limiting}. Moreover, if $f$ is odd, $u_k \geq 0$.
\end{Th}

Then we have the following relation.

\begin{Th}\label{nierownosc}
If (V1) and (F1)--(F4) hold, then either $c$ is critical value of $\cJ$ or $c\geq m_{V_\infty}$.
\end{Th} 

\begin{proof}
    Suppose that the last inequality in (V1) is strict, namely
    \[\liminf_{|x|\to\infty} V(x)>V_\infty\]
    From \cite[Theorem 2.1]{BBJM}, there is a bounded sequence $(u_n) \subset \cN$ such that 
\[\cJ(u_n)\to c\quad \hbox{and} \quad\cJ'(u_n)\to 0.\]
Then, up to a subsequence, $(u_n)$ converges weakly in $X_1^{\cG(K)}$ and strongly in $L^s_{\mathrm{loc}}(\R^N)$, $2 \leq s < 2^*$ to $u$, that is a weak solution of the problem \eqref{eq:singular} with $\varepsilon = 1$. Then by \cite[Corollary 3.32]{JM}, there is a sequence $(z_n)\subset \R^{N-K}$, $\beta >0$ and $R>0$ such that
\begin{equation}\label{liminf}\liminf_{n\to\infty}\int_{B((0,z_n), R)}u_n^2\,dx>\beta.\end{equation}
Now we can distinguish two cases.
\begin{enumerate}
\item[{\textit{Case 1.}}]
If $(z_n)$ contains a bounded subsequence, we can assume that $u_n \weakto u\neq 0$ and $\cJ'(u) = 0$. Moreover for any radius $\rho>0$ by (F4) we have
\begin{equation}\label{na ogr}
\begin{split}
    \cJ(u_n)-\frac12\cJ'(u_n)u_n&=\int_{\R^N}\frac12f(u_n)u_n-F(u_n)\,dx\\
&\geq \int_{B(0,\rho)}\frac12f(u_n)u_n-F(u_n)\,dx\to \int_{B(0,\rho)}\frac12f(u)u-F(u)\,dx
\end{split}
\end{equation}
as $n\to \infty$. Because the left hand side of \eqref{na ogr} converges to $c$ as $n\to \infty$, and $\rho$ is arbitrary, we have
\[c\geq \int_{\R^N}\frac12f(u)u-F(u)\,dx.\]
Since $u \in X_1^{\cG(K)}$ is a critical point of $\cJ$, the right hand side of above inequality equals $\cJ(u)$. Since $u\neq 0$ we obtain that $\cJ(u)=c$ and theorem is proved in this case. 

\item[{\textit{Case 2.}}]Now assume that $(z_n)$ is an unbounded subsequence. Then for any $t>0, \rho>0$,
\begin{align*}
    \cJ(u_n)&\geq \cJ(t u_n)\\
    &=\Phi_{V_\infty} (t u_n)+\int_{B(0,\rho)} \frac12\left(V(x)-V_\infty\right)|t u_n|^2\,dx+\int_{\R^N\setminus B(0,\rho)} \frac12\left(V(x)-V_\infty\right)|t u_n|^2\,dx.
\end{align*}
We can choose $\rho$ so that $V(x)\geq V_\infty$ for all $|x|\geq\rho$. Hence 
\[\cJ( u_n)\geq\Phi_{V_\infty}(t u_n)+\int_{B(0,\rho)} \frac12\left(V(x)-V_\infty\right)|t u_n|^2\,dx.\]
Choose $t:=t_{V_\infty}(u_n)$. Then we obtain
\begin{equation}\label{nier z hat}\cJ( u_n)\geq m_{V_\infty}+\int_{B(0,\rho)} \frac12\left(V(x)-V_\infty\right)|t_{V_\infty}(u_n) u_n|^2\,dx\end{equation}
We claim that the sequence $(t_{V_\infty}(u_n))_n \subset(0,\infty)$ is bounded. Suppose by a contradiction that up to a subsequence $t_{V_\infty}(u_n)\to\infty$. Then by (F3) and (F4)
\[\int_{\R^N}|\nabla u_n|^2+\frac{u_n^2}{|y|^n}+V_\infty u_n^2\,dx=\int_{\R^N}\frac{f(t_{V_\infty}(u_n)u_n)t_{V_\infty}(u_n)u_n}{t_{V_\infty}(u_n)^2}\geq 2\int_{\R^N}\frac{F(t_{V_\infty}(u_n)u_n)}{t_{V_\infty}(u_n)^2}\to \infty\]
as $n\to\infty$. This is impossible since the left hand side of this inequality is bounded.

Suppose that there is a $\gamma>0$ such that
\begin{equation}\label{norma u l2}\|u_n\|_{L^2(B(0,\rho))}\geq \gamma.\end{equation}

Then, as in the case when $(z_n)$ stays bounded, $u_n$ converges, up to a subsequence, weakly in $X_1^{\cG(K)}$ to a nontrivial critical point of $\cJ$ and $\cJ(u)=c$, and the proof is completed.

Hence, assume that \eqref{norma u l2} does not hold. Then up to a subsequence 
\[\|u_n\|_{L^2(B(0,\rho))}\to0\]
as $n\to\infty$.
Then, by \eqref{nier z hat}, we get that $c\geq m_{V_\infty}$ and the proof is completed under a stronger version of (V1).
\end{enumerate}
Now we assume that (V1) holds and then, for $\delta>0$ we have,
\[\liminf_{n\to\infty}V(x)>V_\infty-\delta.\]
By just proved result, either $c$ is critical value of $\cJ$ or $c\geq m_{V_\infty-\delta}$. Suppose that $c$ is not a critical value of $\cJ$. Then by letting $\delta\to 0^+$, by Theorem \ref{thm:cont_dep}, we obtain that $c\geq \lim_{\delta \to 0^+} m_{V_\infty-\delta} = m_{V_\infty}$ and the proof is completed.
\end{proof}

\section{Existence of semiclassical states}

In this section we present the proof of the existence of semiclassical states. We extend the strategy from \cite{Rabinowitz} to a more general class of nonlinear functions $f$ and we estimate the minimal levels on Nehari manifolds instead of mountain pass levels.

\begin{proof}[Proof of Theorem \ref{thm:main1}]
Let $\varepsilon>0$. If $c_\varepsilon$ is not a critical value for $\cJ_\varepsilon$, then by Theorem \ref{nierownosc} \[c_\varepsilon\geq m_{V_\infty}.\]
We will show that this inequality is impossible using a comparison argument.
Let $w$ be the solution of \eqref{eq:limiting} with $k =V_\infty$ such that $\Phi_{V_\infty}(w)=m_{V_\infty}$. Let $R>0$ and $\chi_R\in C^1(\R^+,\R^+)$ be such that $\chi_R(t)=1$ for $t\leq R$, $\chi_R(t)=0$ for $t\geq R+2$, and $|\chi'_R(t)|<1$ for $t\in (R,R+2).$ We also set $v:=\chi_R w$. Then for any $\hat{\theta}>0,$
\[\gamma_R:=\max_{\theta \geq 0}\Phi_{V_\infty}(\theta v)\geq \cJ_\varepsilon(\hat{\theta}v)+\frac12\int_{B(0,R+2)}(V_\infty-V_\varepsilon(x))|\hat{\theta}v|^2\,dx \]
By choosing $\hat{\theta}:=t_{V_\varepsilon}(v)$ we obtain
\[\gamma_R\geq c_\varepsilon+\frac12\int_{B(0,R+2)}(V_\infty-V_\varepsilon)|\hat{\theta}v|^2\,dx.\]
For $\varepsilon$ small enough, $V_\infty-V_\varepsilon(x)\geq \frac12(V_\infty-V(0))$ in $B(0,R+2)$, so we can rewrite above inequality as
\[\gamma_R\geq c_\varepsilon+\frac14\left(V_\infty-V(0)\right)\hat{\theta}^2\int_{B(0,R+2)}v^2\,dx.\]

Note that $\hat{\theta}$ depends on $\varepsilon$ and $R$. We will prove later that
\begin{equation}\label{eq:thetaBoundedAway}
\mbox{there exist } \theta_0>0 \mbox{ such that } \hat{\theta}\geq \theta_0 \mbox{ for sufficiently small } \varepsilon \mbox{ and large }R. 
\end{equation}
For now, assume that \eqref{eq:thetaBoundedAway} holds. Choose $R$ sufficiently large so that
\[\int_{B(0,R+2)}v^2\,dx\geq\frac12\int_{\R^N}w^2\,dx,\]
that gives us
\begin{equation}\label{eq:psi_nier}
\gamma_R\geq c_\varepsilon+\frac18\left(V_\infty-V(0)\right)\hat{\theta}^2\int_{\R^N}w^2\,dx.
\end{equation}
On the other hand, we will show that 
\begin{equation}\label{eq:existsPsi}
\mbox{there is } \psi : (0,\infty) \rightarrow (0, \infty) \mbox{ such that } \psi(R)\to 0\mbox{ as }R\to\infty\mbox{ and } \gamma_R\leq m_{V_\infty}+\psi(R).
\end{equation}
Assuming in addition that \eqref{eq:existsPsi} holds, choosing $R$ so large that
\[\psi(R)<\frac18\left(V_\infty-V(0)\right)\theta_0^2\int_{\R^N}w^2\,dx,\]
so \eqref{eq:psi_nier} implies that $m_{V_\infty}>c_\varepsilon,$ contrary to Theorem \ref{nierownosc}.
To conclude we need to verify \eqref{eq:thetaBoundedAway} and \eqref{eq:existsPsi}. 

To show \eqref{eq:thetaBoundedAway} note that $\hat{\theta}$ is characterized by 
\[\hat{\theta}^2\int_{\R^N}|\nabla v|^2+\frac{u^2}{|y|^2}+V_\varepsilon(x) v^2\,dx=\int_{\R^N}f(\hat{\theta}v)\hat{\theta}v\,dx.\]
From (F1) and (F2) we obtain that for every $\delta>0$ there exists $C_\delta>0$ such that 
\[|f(u)|\leq \delta |u|+ C_\delta |u|^{p-1}.\]
Hence, combining above inequality and (V1), we obtain that
\[\hat{\theta}^2\int_{\R^N}|\nabla v|^2+\frac{v^2}{|y|^2}+V_0 v^2\,dx\leq \int_{\R^N}\delta\hat{\theta}^2v^2+C_\delta|\hat{\theta}v|^{p}\,dx.\]
Choosing $\delta:=\frac{V_0}{2}$ we obtain
\[\hat{\theta}^2\int_{\R^N}|\nabla v|^2+\frac{v^2}{|y|^2}+\frac{V_0}{2} v^2\,dx\leq \int_{\R^N}C_{V_0/2}|\hat{\theta}v|^{p}\,dx.\]
Observe that 
\[\int_{\R^N}|v|^{p}\,dx\leq \int_{\R^N}|w|^{p}\,dx \]
and
\[\int_{\R^N}|\nabla v|^2+\frac{v^2}{|y|^2}+\frac{V_0}{2}v^2\,dx\geq \int_{B(0,R)}|\nabla w|^2+\frac{w^2}{|y|^2}+\frac{V_0}{2}w^2\,dx.\]
For sufficiently large $R$ we have
\[\int_{B(0,R)}|\nabla w|^2+\frac{w^2}{|y|^2}+\frac{V_0}{2}w^2\,dx\geq \frac12\int_{\R^N}|\nabla w|^2+\frac{w^2}{|y|^2}+\frac{V_0}{2}w^2\,dx.\]
Therefore combining above inequalities we obtain
\[\hat{\theta}\geq \left(\frac{\frac12\int_{\R^N}|\nabla w|^2+\frac{w^2}{|y|^2}+\frac{V_0}{2}w^2\,dx}{C_{V_0 / 2}\int_{\R^N}|w|^{p}\,dx}\right)^\frac{1}{p-2}=:\theta_0>0.\]

To show \eqref{eq:existsPsi} note that, from the definition of $\gamma_R$ we have
\[\gamma_R=\Phi_{V_\infty} (t_{V_\infty} (v)v)=m_{V_\infty}+\Phi_{V_\infty}(t_{V_\infty} (\chi_Rw)\chi_Rw)-\Phi_{V_\infty}(w),\]
so we only need to show that
\[|\Phi_{V_\infty}(t_{V_\infty} (\chi_Rw) \chi_Rw)-\Phi_{V_\infty}(w)|\to 0 \hbox{ as } R\to \infty\]
 and then we can just take 
\[
\psi(R):=\Phi_{V_\infty}(t_{V_\infty} (\chi_Rw) \chi_Rw)-\Phi_{V_\infty}(w).
\]
If $R\to\infty$ then $\chi_Rw\to w$ in $X.$ Hence $t_{V_\infty}(\chi_Rw)\to t_{V_\infty}(w) = 1$, which shows the requested convergence.
\end{proof}

\section{Asymptotic analysis}

We start by showing a decay at infinity of solutions to \eqref{eq:singular} and the limiting problem \eqref{eq:limiting}. We follow (with some minor changes) arguments from \cite[Section 6]{BadialeBenciRolando} and we prove the following general result.  

\begin{Th}\label{thm:decay}
Suppose that $V\in \cC(\R^N)$, $\inf V>0$ and (F1)--(F4) hold. Then any nonnegative weak solution $u$ in $X_1$ to \eqref{eq:singular} with $\eps=1$ belongs to $L^\infty (\R^N)$ and satisfies
$$
\limsup_{|x|\to\infty} |x|^\nu u(x) = 0
$$ 
for any $\nu < \frac{N-2+\sqrt{(N-2)^2+4}}{2}$. 
\end{Th}

\begin{proof}
Let $u \geq 0$ be a weak solution to \eqref{eq:singular} with $\eps=1$. Let $1 < a < 2^* - 1$ and let $\varphi \in \cC_0^\infty(\R^N)$ be a nonnegative test function. (F2) implies that we may choose a small radius $r > 0$ such that $|f(\zeta)| \leq \frac{\inf V}{2}  |\zeta|$ for $|\zeta| < r$. Then
\begin{align*}
\int_{\R^N} \nabla u\nabla \varphi \leq \int_{\R^N} \nabla u \nabla \varphi + \frac{u \varphi}{|y|^2} + V(x)u \varphi - f_2 (u) \varphi \, dx &= \int_{\R^N} f_1(u) \varphi \, dx \\ &= \int_{\R^N} \phi(x,u) \varphi \, dx,
\end{align*}
where we set 
\begin{align*}
&f_1 (\zeta) := \chi_{(-r,r)^c} (\zeta) f(\zeta), \quad f_2(\zeta) := f(\zeta) - f_1(\zeta),\\
&\phi(x,\zeta) := f_1 \left( u_k(x)^{(2^* - 1 - a)/(2^* - 1)} |\zeta|^{a/(2^*-1) } \right), \quad (x,\zeta) \in \R^N \times \R.
\end{align*}
Observe that $\phi(x, u(x)) = f_1 (u(x))$, $|f_1 (\zeta)| \lesssim |\zeta|^{2^*-1}$ and hence
$$
\phi(x,\zeta) \lesssim u(x)^{2^*-1-a} |\zeta|^a.
$$
Note that $u^{2^*-1-a} \in L^{2^* / (2^* - 1 - a)} (\R^N)$. Hence, by \cite[Theorem 26]{BadialeBenciRolando}, \cite{Egnell},
$$
\limsup_{|x|\to\infty} |x|^{N-2} u(x) < \infty.
$$
Now, observe that for $\delta\in (0,1)$ we can choose sufficiently large $R>0$ such that $f_1(u)\lesssim |x|^{-4}u$ for $|x|\geq R$, and we may assume that
$f_1(u)\leq \delta |x|^{-2}u$ for $|x|\geq R$. Then
arguing in a similar way as above we show that 
\begin{align*}
\int_{\R^N\setminus B(0,R)}\nabla u\nabla \vp\,dx &\leq \int_{\R^N\setminus B(0,R)}f_1(u)\vp-\frac{u\vp}{|y|^2}\,dx \leq \int_{\R^N\setminus B(0,R)}f_1(u)\vp-\frac{u\vp}{|x|^2}\,dx\\ &
\leq -(1-\delta)  \int_{\R^N\setminus B(0,R)}\frac{u\vp}{|x|^2}\,dx
\end{align*}
for any nonnegative $\vp\in \cC_0^{\infty}(\R^N\setminus \overline{B(0,R)})$.
Since $-\Delta v_\delta = -(1-\delta)|x|^{-2}v_\delta $ we  find a constant $C>0$ such that $Cv_\delta-u\geq 0$ for $|x|\geq R$, where $v_\delta (x):= |x|^{- \frac{N-2+\sqrt{(N-2)^2+4(1-\delta)}}{2}}$, see \cite[Section 6]{BadialeBenciRolando} for details. Therefore
$$
\limsup_{|x|\to\infty} |x|^\nu u(x) <\infty
$$
for any $\nu \leq \frac{N-2+\sqrt{(N-2)^2+4(1-\delta)}}{2}$. Since $\delta$ was arbitrary and $ \frac{N-2+\sqrt{(N-2)^2+4(1-\delta)}}{2}$ is decreasing with respect to $\delta$, the statement holds for all $\nu <  \frac{N-2+\sqrt{(N-2)^2+4}}{2}$. To see that the limit is equal to zero, fix any $\nu <  \frac{N-2+\sqrt{(N-2)^2+4}}{2}$ and choose $\delta > 0$ so small that $\nu + \delta <  \frac{N-2+\sqrt{(N-2)^2+4}}{2}$. Then
$$
\limsup_{|x|\to\infty} |x|^\nu u(x) = \limsup_{|x|\to\infty} |x|^{-\delta} |x|^{\nu + \delta} u(x) = 0.
$$
\end{proof}

\begin{Cor}\label{cor:decay} 
Suppose that $V\in \cC^{\cG(2)}(\R^3)$, $\inf V>0$ and (F1)--(F4) hold.
Suppose that $u\in X_1^{\cG(2)}$ is a nonnegative solution to \eqref{eq:singular} with $\eps = 1$. Then
$$
\U (x) := \frac{u}{\sqrt{x_1^2 + x_2^2}} \left( \begin{array}{c}
-x_2 \\ x_1 \\ 0
\end{array} \right)
$$
is a weak solution to \eqref{eq:maxwell} with $\mu = 1$, that is $\cJ_{curl}'(\U)(\Psi)=0$ for any $\Psi\in\cC_0^{\infty}(\R^3;\R^3)$, where
\begin{equation}\label{eq:Jcurl}
\cJ_{curl}(\U):=\frac12\int_{\R^3}|\curl \U|^2+V(x)|\U|^2 \,dx-\int_{\R^3}G(\U) \, dx.
\end{equation}
Moreover $\U \in L^\infty (\R^3; \R^3)$, $\div(\U)=0$ and we have the following decay
$$
\limsup_{|x|\to\infty} |x|^\nu | \U(x) | = 0 \quad \hbox{for every } \nu <  \frac{N-2+\sqrt{(N-2)^2+4}}{2}.
$$
\end{Cor}
\begin{proof}
The equivalence result for problems \eqref{eq:curlcurl} and \eqref{eq:singular} has been obtained in \cite[Theorem 2.1]{GMS} for the case $V\equiv 0$. By the inspection of the proof, we easily conclude that $\U$ is a weak solution to  \eqref{eq:curlcurl}  and $\cJ(u)=\cJ_{curl}(\U)$, cf. \cite{BB,BMS}. Decay properties follow from Theorem \ref{thm:decay}.
\end{proof}

Observe that (V1) implies that $V \in L^q_{\mathrm{loc}} (\R^N)$ for any $q \geq 1$. Moreover, from (V1) and (V2) we get that $V \in L^\infty (\R^N)$ and $X_\varepsilon^{\cG(K)} = Y^{\cG(K)}$. From now on we again assume that $V\in \cC^{\cG(K)}(\R^N)$, $N>K\geq 2$.

\begin{Lem} \label{ineq1}
Suppose that (V1), (V3), (F1)--(F4) hold and $f$ is odd. Then $\limsup_{\varepsilon \to 0^+} c_{\varepsilon} \leq m_{V_\infty}$.
\end{Lem}

\begin{proof}
Let $u_{0} \in Y^{\cG(K)}$ be a nonnegative weak solution to \eqref{eq:limiting} with $k=V_\infty$ such that $\Phi_{V_\infty} (u_0) = m_{V_\infty}$. Observe that (V1) implies that for any $\delta > 0$ there is $M = M_\delta$ such that
$$
V(x) \geq V_\infty - \delta \quad \mbox{for } |x| \geq M. 
$$
Hence
$$
 \int_{|x| \geq M/\varepsilon} (V_\infty-V_\varepsilon (x)) u_0^2 \, dx \leq \delta \int_{\R^N} u_0^2 \, dx. 
$$
On the other hand
\begin{align*}
 \int_{|x| < M/\varepsilon} (V_\infty-V_\varepsilon (x)) u_0^2 \, dx  &=    \int_{|x| < M} \varepsilon^{-N} (V_\infty-V (x)) u_0(x/\varepsilon)^2 \, dx \\
&\geq \int_{|x| < M} \varepsilon^{-N} (V(0)-V (x)) u_0(x/\varepsilon)^2 \, dx.
\end{align*}
Note that, thanks to Theorem \ref{thm:decay},
\begin{align*}
&\quad \left| \int_{|x| < M} \varepsilon^{-N} (V(0) -V (x)) u_0(x/\varepsilon)^2 \, dx \right| \\
&= \left|  \int_{|x| < M} \varepsilon^{-N+2\nu} (V(0)-V (x)) \left( u_0(x/\varepsilon) (|x|/\varepsilon)^\nu \right)^2 |x|^{-2\nu} \, dx\right| \\
&\lesssim \int_{|x| < M} \varepsilon^{-N+2\nu} |V(0)-V (x)|  |x|^{-2\nu} \, dx = \int_{|x| < M} \varepsilon^{-N+2\nu} |V(0)-V (x)|  |x|^{N-2\nu} |x|^{-N} \, dx \\
&\lesssim \int_{|x| < M} \varepsilon^{-N+2\nu} |x|^{-N} \, dx \lesssim \varepsilon^{-N+2\nu} \to 0 \quad \mbox{ as } \varepsilon \to 0^+,
\end{align*}
where $\nu \in \left( \frac{N}{2}, \frac{N-2+\sqrt{(N-2)^2+4}}{2} \right)$ is chosen, thanks to (V3), so that 
$$
\limsup_{|x| \to 0} |V(0) - V(x)| |x|^{N-2\nu} < \infty.
$$
Thus
$$
\liminf_{\varepsilon \to 0^+} \int_{|x| < M/\varepsilon} (V_\infty-V_\varepsilon (x)) u_0^2 \, dx \geq \liminf_{\varepsilon \to 0^+} \int_{|x| < M} \varepsilon^{-N} (V(0)-V (x)) u_0(x/\varepsilon)^2 \, dx = 0.
$$
Hence
$$
\liminf_{\varepsilon \to 0^+} \int_{\R^N} (V_\infty-V_\varepsilon (x)) u_0^2 \, dx \geq -\delta \int_{\R^N} u_0^2 \, dx.
$$
Since $\delta > 0$ was arbitary, we get 
$$
\liminf_{\varepsilon \to 0^+} \int_{\R^N} (V_\infty-V_\varepsilon (x)) u_0^2 \, dx \geq 0
$$
or, equivalently,
\begin{equation}\label{eq:limsup}
\limsup_{\varepsilon \to 0^+} \int_{\R^N} (V_\varepsilon (x)-V_\infty) u_0^2 \, dx \leq 0
\end{equation}
We note also that $t_{V_\varepsilon}(u_0)$ stays bounded as $\varepsilon \to 0^+$. Indeed, denote $t_\varepsilon := t_{V_\varepsilon} (u_0)$ and suppose that $t_\varepsilon \to \infty$. From Fatou's lemma and (F3) we have that
$$
\int_{\R^N} |\nabla u_0|^2 + \frac{u_0^2}{|y|^2} + V_\varepsilon(x) u_0^2 \, dx = \int_{\R^N} \frac{f( t_\varepsilon u_0 ) t_\varepsilon u_0}{t_\varepsilon^2} \, dx \geq 2 \int_{\R^N} \frac{F( t_\varepsilon u_0 )}{t_\varepsilon^2} \, dx \to \infty.
$$
Hence
$$
\int_{\R^N} V_\varepsilon(x) u_0^2 \, dx \to \infty,
$$
which is a contradiction with \eqref{eq:limsup}. Thus $(t_\varepsilon)$ is bounded and then
\begin{align*}
\limsup_{\varepsilon \to 0^+} c_{\varepsilon}  &= \limsup_{\varepsilon \to 0^+} \inf_{\cN_{\varepsilon}} \cJ_{\varepsilon} \leq  \limsup_{\varepsilon \to 0^+} \cJ_{\varepsilon} ( t_\varepsilon  u_0 ) \\
&\leq  \limsup_{\varepsilon \to 0^+} \Phi_{V_\infty} ( t_\varepsilon  u_0 ) +  \limsup_{\varepsilon \to 0^+}  \int_{\R^N} (V_\varepsilon (x)-V_\infty) t_\varepsilon^2 u_0^2 \, dx  \\
&= \limsup_{\varepsilon \to 0^+} \Phi_{V_\infty} ( t_\varepsilon u_0 ) \leq \Phi_{V_\infty} (u_0) = m_{V_\infty}.
\end{align*}
\end{proof}

\begin{Lem} \label{ineq2}
Suppose that (V1), (F1)--(F4) hold. Then $c_\varepsilon \geq m_{V_0}$.
\end{Lem}

\begin{proof}
Let $u_\varepsilon \in Y^{\cG(K)}$ be a weak solution to \eqref{eq:singular} such that $\cJ_\varepsilon (u_\varepsilon) = c_\varepsilon$. Then
$$
m_{V_0} = \inf_{\cM_{V_0}} \Phi_{V_0} \leq \Phi_{V_0} ( t_{V_0} (u_\varepsilon) u_\varepsilon ) \leq \cJ_\varepsilon ( t_{V_0} (u_\varepsilon) u_\varepsilon ) \leq \cJ_\varepsilon (u_\varepsilon) = c_\varepsilon.
$$
\end{proof}

In what follows, we will consider $(c_\varepsilon)$ and $(u_\varepsilon)$ as sequences, without writing $\varepsilon_n$, always passing to a subsequence with respect to $\varepsilon$ (if needed). 

\begin{Lem}
Suppose that (V1), (F1)--(F4) hold. The sequence $(u_\varepsilon)$ is bounded in $Y^{\cG(K)}$.
\end{Lem}

\begin{proof}
Recall that $\cJ_\varepsilon '(u_\varepsilon)(u_\varepsilon) = 0$, thus
\begin{align*}
\| u_\varepsilon \|^2 &\lesssim \int_{\R^N} |\nabla u_\varepsilon |^2 + \frac{u_\varepsilon^2}{|y|^2} + V_0 u_\varepsilon^2 \, dx \leq \int_{\R^N} |\nabla u_\varepsilon |^2 + \frac{u_\varepsilon^2}{|y|^2} + V_\varepsilon(x) u_\varepsilon^2 \, dx = \int_{\R^N} f(u_\varepsilon) u_\varepsilon \, dx
\\
&\leq \delta |u_\varepsilon|_2^2 + C_\delta |u_\varepsilon|_p^p \lesssim \delta \|u\|_Y^2 + C_\delta \|u\|_Y^p.
\end{align*}
Choosing sufficiently small $\delta$ we obtain that $\|u_\varepsilon\|_Y \lesssim 1$.
\end{proof}

Using \cite[Theorem 4.7]{BMS} we obtain that there are $(\widetilde{u}_i) \subset Y^{\cG(K)}$, $(z_\varepsilon^i) \subset \R^{N-K}$ such that $z_\varepsilon^0 = 0$, $|z_\varepsilon^i - z_\varepsilon^j| \to \infty$ for $i \neq j$, and (passing to a subsequence)
\begin{align}
\nonumber
&u_\varepsilon (\cdot + (0, z_\varepsilon^i) ) \weakto \widetilde{u}_i \hbox{ in } Y^{\cG(K)}; \\
\label{splitting:norm}
&\lim_{\varepsilon \to 0^+}\int_{\R^N} |\nabla u_\varepsilon|^2 + \frac{u_\varepsilon^2}{|y|^2} \, dx  = \sum_{j=0}^i \int_{\R^N} |\nabla \widetilde{u}_j|^2 + \frac{\widetilde{u}_j^2}{|y|^2} \, dx + \lim_{\varepsilon \to 0^+} \int_{\R^N} |\nabla v_\varepsilon^i|^2 + \frac{(v_\varepsilon^i)^2}{|y|^2} \, dx,
\end{align}
where $v_\varepsilon^i := u_\varepsilon - \sum_{j=0}^i \widetilde{u}_j ( \cdot - (0, z_\varepsilon^j) )$ and
\begin{align}
\label{splitting:f}
&\limsup_{\varepsilon \to 0^+} \int_{\R^N} f(u_\varepsilon)u_\varepsilon \, dx = \sum_{j=0}^\infty \int_{\R^N} f(\widetilde{u}_j) \widetilde{u}_j \, dx, \\
\label{splitting:F}
&\limsup_{\varepsilon \to 0^+} \int_{\R^N} F(u_\varepsilon) \, dx = \sum_{j=0}^\infty \int_{\R^N} F(\widetilde{u}_j) \, dx.
\end{align}

\begin{Lem}\label{lem:split-crit-points}
Suppose that (V1), (V2), (F1)--(F4) hold. For every $i \geq 0$, either $\widetilde{u}_i$ is a critical point of $\Phi_{V(0,z_i)}$ for some $z_i \in \R^{N-K}$, or is a critical point of $\Phi_{V_\infty}$. Moreover, for $i = 0$, $z_i = 0$ and $\widetilde{u}_0$ is a critical point of $\Phi_{V(0)}$.
\end{Lem}

\begin{proof}
Since $\cJ_\varepsilon '(u_\varepsilon) = 0$ and $u_\varepsilon (\cdot + (0, z_\varepsilon^i) ) \weakto \widetilde{u}_i$ we observe that
\begin{align*}
0 &= \cJ_\varepsilon'(u_\varepsilon) (\varphi (\cdot - (0, z_\varepsilon^i)))\\
&= \int_{\R^N} \nabla u_\varepsilon \nabla \varphi(\cdot - (0, z_\varepsilon^i)) + \frac{u_\varepsilon \varphi(\cdot - (0, z_\varepsilon^i))}{|y|^2} + V_\varepsilon(x) u_\varepsilon \varphi(\cdot - (0, z_\varepsilon^i)) \, dx \\
&\quad - \int_{\R^N} f(u_\varepsilon) \varphi(\cdot - (0, z_\varepsilon^i)) \, dx \\
&= \int_{\R^N} \nabla u_\varepsilon (\cdot + (0, z_\varepsilon^i)) \nabla \varphi + \frac{u_\varepsilon (\cdot + (0, z_\varepsilon^i)) \varphi}{|y|^2} + V_\varepsilon(\cdot + (0, z_\varepsilon^i)) u_\varepsilon (\cdot + (0, z_\varepsilon^i)) \varphi \, dx\\
&\quad - \int_{\R^N} f(u_\varepsilon(\cdot + (0, z_\varepsilon^i))) \varphi \, dx.
\end{align*}
Weak convergence and compact embeddings $Y^{\cG(K)} \subset L^s_{\mathrm{loc}}(\R^N)$, $2 \leq s < 2^*$ imply that
\begin{multline*}
\int_{\R^N} \nabla u_\varepsilon (\cdot + (0, z_\varepsilon^i)) \nabla \varphi + \frac{u_\varepsilon (\cdot + (0, z_\varepsilon^i)) \varphi}{|y|^2} \, dx - \int_{\R^N} f(u_\varepsilon(\cdot + (0, z_\varepsilon^i))) \varphi \, dx \\
\to \int_{\R^N} \nabla \widetilde{u}_j \nabla \varphi + \frac{\widetilde{u}_j \varphi}{|y|^2} \, dx - \int_{\R^N} f(\widetilde{u}_j) \varphi \, dx.
\end{multline*}
Now we consider
$$
\int_{\R^N} V_\varepsilon(\cdot + (0, z_\varepsilon^i)) u_\varepsilon (\cdot + (0, z_\varepsilon^i)) \varphi \, dx = \int_{\R^N} V(\varepsilon x+ (0, \varepsilon z_\varepsilon^i)) u_\varepsilon (\cdot + (0, z_\varepsilon^i)) \varphi \, dx.
$$
If $\limsup_{\varepsilon \to 0^+} |\varepsilon z_\varepsilon^i| < \infty$, we may assume that $\varepsilon z_\varepsilon^i \to z_i$ for some $z_i \in \R^{N-K}$. Hence
$$
V(\varepsilon x+ (0, \varepsilon z_\varepsilon^i)) u_\varepsilon (x + (0, z_\varepsilon^i)) \varphi(x) \to V(0, z_i) \widetilde{u}_i(x) \varphi(x) \mbox{ for a.e. } x \in \R^N.
$$
Thanks to (V1) and (V2), $V \in L^\infty (\R^N)$, and for any measurable $E \subset \supp \varphi$ we get
$$
\left| \int_{\R^N} V(\varepsilon x+ (0, \varepsilon z_\varepsilon^i)) u_\varepsilon (\cdot + (0, z_\varepsilon^i)) \varphi \, dx \right| \leq |V|_\infty |u_\varepsilon|_2 |\varphi \chi_E|_2 \lesssim |\varphi \chi_E|_2^2.
$$
From Vitali convergence theorem we get
$$
\int_{\R^N} V(\varepsilon x+ (0, \varepsilon z_\varepsilon^i)) u_\varepsilon (\cdot + (0, z_\varepsilon^i)) \varphi \, dx \to \int_{\R^N} V(0, z_i) \widetilde{u}_i \varphi \,dx
$$
and $\Phi_{V(0,z_i)} '(\widetilde{u}_i)(\varphi) = 0$. Suppose now that $(\varepsilon z_\varepsilon^i)$ is unbounded. Up to a subsequence, we assume that $|\varepsilon z_\varepsilon^i| \to \infty$. Then
$$
V(\varepsilon x+ (0, \varepsilon z_\varepsilon^i)) u_\varepsilon (x + (0, z_\varepsilon^i)) \varphi(x) \to V_\infty \widetilde{u}_i(x) \varphi(x) \mbox{ for a.e. } x \in \R^N
$$
and the same reasoning shows that $\Phi_{V_\infty}' (\widetilde{u}_i)(\varphi) = 0$. The statement for $i = 0$ follows simply from the fact that $z_\varepsilon^0 = 0$.
\end{proof}

We will show that only finite number of $\widetilde{u}_i$ is nonzero and at least one of them is nonzero. For this purpose we put $I := \{ i \ : \ \widetilde{u}_i \neq 0 \}$.

\begin{Lem}
Suppose that (V1), (V2), (F1)--(F4) hold. There holds $I \neq \emptyset$ and $\# I < \infty$. 
\end{Lem}

\begin{proof}
We start by showing that $\# I < \infty$. Note that, using \eqref{splitting:f}, we get
\begin{align*}
&\quad 1 \gtrsim \limsup_{\varepsilon \to 0^+} \|u_\varepsilon\|_Y^2 \gtrsim  \limsup_{\varepsilon \to 0^+} \int_{\R^N} |\nabla u_\varepsilon|^2 + \frac{u_\varepsilon^2}{|y|^2} + V_\varepsilon(x) u_\varepsilon^2 \, dx = \limsup_{\varepsilon \to 0^+} \int_{\R^N} f(u_\varepsilon)u_\varepsilon \, dx \\
&=  \sum_{j \in I} f(\widetilde{u}_j) \widetilde{u_j} = \sum_{j \in I} \int_{\R^N} |\nabla \widetilde{u}_j|^2 + \frac{\widetilde{u}_j^2}{|y|^2} + k_j \widetilde{u}_j^2 \, dx \geq \sum_{j \in I} \int_{\R^N} |\nabla \widetilde{u}_j|^2 + \frac{\widetilde{u}_j^2}{|y|^2} + V_0 \widetilde{u}_j^2 \, dx \\
&\gtrsim \sum_{j \in I} \| \widetilde{u}_j \|_Y^2 \geq \sum_{j \in I} \left( \inf_{\cM_{k_j}} \| \cdot \|_Y^2 \right),
\end{align*}
where $k_j = V(0, z_j)$ for some $z_j \in \R^{N-K}$ or $k_j =V_\infty$. We claim that 
$$
\inf_{\cM_{k_j}} \| \cdot \|_Y \gtrsim 1.
$$
Indeed, note that for $u \in \cM_{k_j}$ we get
$$
\|u\|_Y^2 \lesssim \int_{\R^N} |\nabla u|^2 + \frac{u^2}{|y|^2} + k_j u^2 \, dx = \int_{\R^N} f(u)u \, dx \lesssim \varepsilon \|u\|_Y^2 + C_\varepsilon \|u\|_Y^p,
$$
where the estimates are independend on $k_j$, because $V_0 \leq k_j \leq |V|_\infty$. Hence the set $I$ must be finite; otherwise $\sum_{j \in I} \left( \inf_{\cM_{k_j}} \| \cdot \|_Y \right) = \infty$.  

Now we need to show that $I \neq \emptyset$. Suppose by contradiction that $I = \emptyset$. Then $\int_{\R^N} f(u_\varepsilon) u_\varepsilon \, dx \to 0$, but it contradicts the inequality
$$
\inf_{\cN_\varepsilon} \|\cdot\|_Y \leq \|u_\varepsilon\|_Y \lesssim \left( \int_{\R^N} f(u_\varepsilon) u_\varepsilon \, dx \right)^{1/2},
$$
since, as above, we can show that $\inf_{\cN_\varepsilon} \|\cdot\|_Y \gtrsim 1$. Hence $I \neq \emptyset$.
\end{proof}

Since we already know that $I$ is finite, we get (cf. \cite[formula (1.11)]{JM})
\begin{equation}\label{Lp-convergence}
\left| u_\varepsilon - \sum_{j \in I} \widetilde{u}_j (\cdot - (0,z_\varepsilon^j)) \right|_p \to 0 \quad \mbox{as } \varepsilon \to 0^+.
\end{equation}

\begin{Lem}\label{lem:decomp_final}
Suppose that (V1)--(V3), (F1)--(F4) hold and $f$ is odd. There holds
\begin{align*}
&\left\| u_\varepsilon - \sum_{j \in I} \widetilde{u}_j (\cdot - (0,z_\varepsilon^j)) \right\|_Y \to 0, \\
&\cJ_\varepsilon (u_\varepsilon) \to \sum_{j \in I} \Phi_{k_j} (\widetilde{u}_j),
\end{align*}
and
$$
\# I \leq \frac{m_{V_\infty}}{m_{V_0}}.
$$
\end{Lem}

\begin{proof}
Put
$$
v_\varepsilon := u_\varepsilon - \sum_{j \in I} \widetilde{u}_j (\cdot - (0,z_\varepsilon^j)).
$$
It is clear the $(v_\varepsilon)$ is bounded in $Y^{\cG(K)}$. Moreover
\begin{align*}
&\cJ_\varepsilon ' (u_\varepsilon) (v_\varepsilon) = 0, \\
&\Phi_{k_j} ' (\widetilde{u}_j) (v_\varepsilon(\cdot + (0,z_\varepsilon^j)) = 0.
\end{align*}
Then
\begin{align*}
\| v_\varepsilon \|_Y^2 &\lesssim  \int_{\R^N} |\nabla v_\varepsilon|^2 + \frac{v_\varepsilon^2}{|y|^2} + V_\varepsilon (x) v_\varepsilon^2 \, dx =: \langle v_\varepsilon, v_\varepsilon \rangle_{V_\varepsilon} = \int_{\R^N} f(u_\varepsilon) v_\varepsilon \, dx + \langle v_\varepsilon - u_\varepsilon, v_\varepsilon \rangle_{V_{\varepsilon}} \\
&= \int_{\R^N} f(u_\varepsilon) v_\varepsilon \, dx + \sum_{j \in I} \langle \widetilde{u}_j (\cdot - (0, z_\varepsilon^j)), v_\varepsilon \rangle_{V_{\varepsilon}}.
\end{align*}
Observe that
\begin{align*}
\langle \widetilde{u}_j (\cdot - (0, z_\varepsilon^j)), v_\varepsilon \rangle_{V_{\varepsilon}} &= - \int_{\R^N} f(\widetilde{u}_j (\cdot - (0,z_\varepsilon^j))) v_\varepsilon \, dx - \Phi_{k_j} ' (\widetilde{u}_j) (v_\varepsilon (\cdot + (0,z_\varepsilon^j))) \\
&\quad - \int_{\R^N} (k_j - V_\varepsilon(x)) \widetilde{u}_j (\cdot - (0,z_\varepsilon^j)) v_\varepsilon \, dx.
\end{align*}
Then, from the H\"older inequality, (F1), (F2) and \eqref{Lp-convergence},
\begin{align*}
\limsup_{\varepsilon \to 0^+} \left| \int_{\R^N} f(\widetilde{u}_j (\cdot - (0,z_\varepsilon^j))) v_\varepsilon \, dx \right|  &\lesssim \limsup_{\varepsilon \to 0^+} \left( \delta |\widetilde{u}_j|_2 |v_\varepsilon|_2 + C_\delta |\widetilde{u}_j|_p^{p-1} |v_\varepsilon|_p \right) \\
&\lesssim \limsup_{\varepsilon \to 0^+} \left( \delta \|v_\varepsilon\|_Y^2 + C_\delta |v_\varepsilon|_p \right) = \limsup_{\varepsilon \to 0^+} \delta \|v_\varepsilon\|_Y^2 \lesssim \delta
\end{align*}
for any $\delta > 0$, and therefore 
$$
\int_{\R^N} f(\widetilde{u}_j (\cdot - (0,z_\varepsilon^j))) v_\varepsilon \, dx \to 0.
$$
Moreover, as in the proof of Lemma \ref{lem:split-crit-points},
$$
\int_{\R^N} (k_j - V_\varepsilon(x)) \widetilde{u}_j (\cdot - (0,z_\varepsilon^j)) v_\varepsilon \, dx = \int_{\R^N} (k_j - V(\varepsilon x + (0,\varepsilon z_\varepsilon^j))) \widetilde{u}_j  v_\varepsilon (\cdot + (0,z_\varepsilon^j)) \, dx \to 0.
$$
Hence 
$$
\| v_\varepsilon \|^2 \lesssim \int_{\R^N} f(u_\varepsilon) v_\varepsilon \, dx + o(1).
$$
As before, using \eqref{Lp-convergence}, we get also that
$$
\int_{\R^N} f(u_\varepsilon) v_\varepsilon \, dx \to 0
$$
and $\| v_\varepsilon \| \to 0$. To complete the proof, taking into account \eqref{splitting:F}, it is enough to show that
$$
\int_{\R^N} |u_\varepsilon|^2 + \frac{u_\varepsilon^2}{|y|^2} + V_\varepsilon(x) u_\varepsilon^2 \, dx \to \sum_{j \in I} \int_{\R^N} |\widetilde{u}_j|^2 + \frac{\widetilde{u}_j^2}{|y|^2} + k_j \widetilde{u}_j^2 \, dx.
$$
Note that \eqref{splitting:norm} implies that
\begin{align*}
\lim_{\varepsilon \to 0^+} \int_{\R^N} |u_\varepsilon|^2 + \frac{u_\varepsilon^2}{|y|^2} \, dx = \sum_{j \in I} \int_{\R^N} |\widetilde{u}_j|^2 + \frac{\widetilde{u}_j^2}{|y|^2} \, dx + \underbrace{\lim_{\varepsilon \to 0^+} \int_{\R^N} |v_\varepsilon|^2 + \frac{v_\varepsilon^2}{|y|^2} \, dx}_{=0},
\end{align*}
hence we only need to show that
$$
\int_{\R^N} V_\varepsilon(x) u_\varepsilon^2 \, dx \to \sum_{j \in I} \int_{\R^N} k_j \widetilde{u}_j^2 \, dx.
$$
For this purpose, we note first that $\| v_\varepsilon \|_Y \to 0$ implies then that
$$
\int_{\R^N} V_\varepsilon(x) \left( u_\varepsilon - \sum_{j \in I} \widetilde{u}_j (\cdot - (0,z_\varepsilon^j)) \right)^2 \, dx \to 0.
$$
Hence
\begin{align*}
\int_{\R^N} V_\varepsilon(x) u_\varepsilon^2 \, dx &= 2 \sum_{j \in I} \int_{\R^N} V_\varepsilon(x) u_\varepsilon \widetilde{u}_j (\cdot - (0,z_\varepsilon^j)) \, dx - \sum_{i \neq j} \int_{\R^N} V_\varepsilon(x) \widetilde{u}_i (\cdot - (0,z_\varepsilon^i)) \widetilde{u}_j (\cdot - (0,z_\varepsilon^j)) \, dx \\
&\quad - \sum_{j \in I} \int_{\R^N} V_\varepsilon(x) | \widetilde{u}_j (\cdot - (0,z_\varepsilon^j)) |^2 \, dx + o(1).
\end{align*}
Note that for $i \neq j$ we have $|z_\varepsilon^j - z_\varepsilon^i| \to \infty$ and
$$
\left| \int_{\R^N} V_\varepsilon(x) \widetilde{u}_i (\cdot - (0,z_\varepsilon^i)) \widetilde{u}_j (\cdot - (0,z_\varepsilon^j)) \, dx \right| \lesssim \int_{\R^N}| \widetilde{u}_i \widetilde{u}_j (\cdot - (0,z_\varepsilon^j-z_\varepsilon^i)) | \, dx \to 0
$$
from Vitali convergence theorem. Then, similarly as in Lemma \ref{lem:split-crit-points},
\begin{align*}
&\int_{\R^N} V_\varepsilon(x) u_\varepsilon \widetilde{u}_j (\cdot - (0,z_\varepsilon^j)) \, dx = \int_{\R^N} V(\varepsilon x + (0, \varepsilon z_\varepsilon^j)) u_\varepsilon (\cdot + (0,z_\varepsilon^j)) \widetilde{u}_j  \, dx \to \int_{\R^N} k_j  \widetilde{u}_j^2 \, dx, \\
&\int_{\R^N} V_\varepsilon(x) | \widetilde{u}_j (\cdot - (0,z_\varepsilon^j)) |^2 \, dx = \int_{\R^N} V(\varepsilon x + (0,\varepsilon z_\varepsilon^j)) \widetilde{u}_j^2 \, dx\to \int_{\R^N} k_j  \widetilde{u}_j^2 \, dx.
\end{align*}
Hence
$$
\int_{\R^N} V_\varepsilon(x) u_\varepsilon^2 \, dx \to \sum_{j \in J} \int_{\R^N} k_j \widetilde{u}_j^2 \, dx.
$$

As a corollary of Lemma \ref{ineq1} and Lemma \ref{ineq2}, we get that $\limsup_{\varepsilon \to 0^+} c_\varepsilon \in [ m_{V_0}, m_{V_\infty} ]$. Therefore we may assume that, up to a subsequence, $c_\varepsilon \to m$, where $m \in [ m_{V_0}, m_{V_\infty} ]$.

To show that $\# I \leq \frac{m_{V_\infty}}{m_{V_0}}$ observe that
\begin{equation}\label{c-ineq}
m_{V_\infty} \geq m = \lim_{\varepsilon \to 0^+} \cJ_\varepsilon (u_\varepsilon) = \sum_{j \in I} \Phi_{k_j} (\widetilde{u}_j) \geq \sum_{j \in I} m_{k_j} \geq m_{V_0} \# I.
\end{equation}
\end{proof}

\begin{Rem}\label{rem1}
Observe that, if at least one of $k_j =V_\infty$, then $\#I = 1$ and
$$
u_\varepsilon - U(\cdot - (0, z_\varepsilon)) \to 0
$$
for some weak solution $U$ of \eqref{eq:limiting} with $k =V_\infty$. In this case $|\varepsilon z_\varepsilon| \to \infty$. Indeed, if at least one of $k_j =V_\infty$, from \eqref{c-ineq} we get
$$
m_{V_\infty} \geq \sum_{j \in I} m_{k_j} \geq m_{V_\infty} + (\#I - 1) m_{V_0}
$$
and $\# I = 1$.
\end{Rem}

\begin{proof}[Proof of Theorem \ref{thm:main2}]
The statement follows directly from Lemma \ref{lem:decomp_final} and Remark \ref{rem1}.
\end{proof}

\begin{proof}[Proof of Theorem \ref{thm:main3}]
	The statement follows directly from Theorem \ref{thm:main2} and Corollary \ref{cor:decay}. 
\end{proof}

\section*{Acknowledgement}

Bartosz Bieganowski, Adam Konysz and Jarosław Mederski were partly supported by the National Science Centre, Poland (Grant No. 2017/26/E/ST1/00817).

\end{document}